\documentclass[11pt]{article}

\usepackage{amsmath,amssymb,graphicx,mathrsfs,hyperref,amsthm,color}

\usepackage{amsfonts}

\numberwithin{equation}{section}
\newtheorem{theorem}{Theorem}[section]
\newtheorem{corollary}[theorem]{Corollary}
\newtheorem{lemma}[theorem]{Lemma}
\newtheorem{proposition}[theorem]{Proposition}
\theoremstyle{definition}

\newtheorem{remark}[theorem]{Remark}
\newtheorem{example}[theorem]{Example}
\newcommand{\Cinf}{\ensuremath{\mathcal{C}^\infty}}

\newcommand{\mb}[1]{\ensuremath{\mathbb{#1}}}
\newcommand{\N}{\mb{N}}

\newcommand{\R}{\mb{R}}

\newcommand{\G}{\ensuremath{{\mathcal G}}}

\newcommand{\EM}{\ensuremath{{\mathcal E}_{M}}}

\newcommand{\Neg}{\mathcal{N}}

\newcommand{\supp}{\mathop{\mathrm{supp}}}

\newfont{\bigmath}{cmr12 at 13pt}

\newfont{\grecomath}{cmmi12 at 15pt}

\newcommand{\p}{\ensuremath{\partial}}
\newcommand{\beq}{\begin{equation}}
\newcommand{\eeq}{\end{equation}}

\newcommand{\eps}{\varepsilon}

\newcommand{\cB}{{\mathcal B}}
\newcommand{\cC}{{\mathcal C}}
\newcommand{\cD}{{\mathcal D}}
\newcommand{\cE}{{\mathcal E}}

\newcommand{\cG}{{\mathcal G}}

\newcommand{\cM}{{\mathcal M}}
\newcommand{\cN}{{\mathcal N}}

\newcommand{\cP}{{\mathcal P}}

\newcommand{\cV}{{\mathcal V}}

\begin{document}
\title{Regular generalized solutions to semilinear wave equations}
\author
{Hideo Deguchi\\
   \texttt{hdegu@sci.u-toyama.ac.jp}\\
   Department of Mathematics, University of Toyama\\
   Gofuku 3190, 930-8555 Toyama, Japan
\and
Michael Oberguggenberger\\
  \texttt{michael.oberguggenberger@uibk.ac.at}\\
Arbeitsbereich f\"{u}r Technische Mathematik,
Universit\"{a}t Innsbruck \\
Technikerstra\ss e 13, A-6020 Innsbruck, Austria
}
\date{}
\maketitle

\begin{abstract}
The paper is devoted to proving an existence and uniqueness result for generalized solutions to semilinear wave equations with a small nonlinearity in space dimensions 1, 2, 3. The setting is the one of Colombeau algebras of generalized functions. It is shown that for a nonlinearity of arbitrary growth and sign, but multiplied with a small parameter, the initial value problem for the semilinear wave equation has a unique solution in the Colombeau algebra of generalized functions of bounded type.
The proof relies on a fixed point theorem in the ultra-metric topology on the algebras involved. In classical terms, the result says that the semilinear wave equations under consideration have global classical solutions up to a rapidly vanishing error.
\end{abstract}



\section{Introduction}\label{sec : intro}

This paper addresses existence and regularity of solutions to semilinear wave equations with a small nonlinearity. The equations are of the form
\begin{equation}\label{eq : NLWE}
\begin{array}{l}
	\partial_t^2u - \Delta u = h(\varepsilon) f(u),  \quad t \in [0,T], \ x \in \mathbb{R}^d, \vspace{4pt}\\
	u|_{t = 0} = u_{0},\quad \partial_tu|_{t = 0} = u_{1}, \quad x \in \mathbb{R}^d
\end{array}
\end{equation}
where $\Delta$ denotes the Laplacian, $h(\varepsilon)$ is a net of positive real numbers tending to zero as $\varepsilon\to 0$, $f$ is smooth with
$f(0) = 0$. The initial data $u_0$ and $u_1$ are generalized functions of compact support, and the space dimension is $d=1,2,3$.
Approximating the initial data by nets of smooth functions $(u_{0\varepsilon}, u_{1\varepsilon})_{\varepsilon\in(0,1]}$, the goal of the paper is to establish the existence of a net of smooth solutions $(u_\varepsilon)_{\varepsilon\in(0,1]}$ up to an asymptotic error term of $O(\varepsilon^\infty)$. The paper is formulated in the framework of Colombeau generalized functions.

We present a new existence result of a global generalized solution without growth or sign restrictions on the nonlinearity $f$, for initial data possessing so-called $\mathcal{G}^0$-regularity. It is motivated by a result on propagation of singularities in the one-dimensional case of the authors \cite{DO:2019}. Our main tool will be the Banach fixed point theorem in the so-called sharp topology, a complete ultra-metric topology on the Colombeau algebras.

In the classical literature, the semilinear wave equation\;\eqref{eq : NLWE} has been studied intensively when $h(\varepsilon)\equiv 1$, with small or with large initial data. The existence or nonexistence of a global classical solution is known to depend on the space dimension, the sign and the growth properties of $f(u)$ as $|u|\to \infty$, and the size of the initial data. Just to pick two prominent examples: When $f(u) = -|u|^{p-1}u$, large data solutions of finite energy exist for all $p\geq 1$ in space dimension $d=1$, whereas the critical exponent for such solutions is $p=5$ in space dimension $d=3$, see e.~g. \cite{Haraux:2018, Struwe:1992}. On the other side, when $f(u) = |u|^{p}$, small data solutions blow up for small $p$ less than the Strauss exponent but exist globally beyond that exponent \cite{LY:2018, TW:2014}. For weak solutions, the question of existence and/or uniqueness is yet different \cite{Struwe:1992}. This brief recall of the huge classical literature on semilinear wave equations may suffice. Our focus is on solutions in the Colombeau algebra $\cG([0,T]\times\R^d)$.

Before describing our results in more detail, a few words about Colombeau algebras are in order. Let $\Omega$ be an open subset of $\R^d$. Denote by $\mathcal{O}_M(\R)$ the space of smooth functions such that each derivative grows at most polynomially at infinity, and by $\mathcal{D}'(\Omega)$ the space of distributions on $\Omega$. Colombeau algebras are algebras of families $(u_\varepsilon)_{\varepsilon\in(0,1]}$ of smooth functions modulo asymptotically vanishing families, i.e., families all whose derivatives vanish of order $\varepsilon^\infty$ (i.~e., of order $\varepsilon^a$ for every $a\geq 0$) on compact sets as $\varepsilon\to 0$. For the following discussion, we shall need
\begin{itemize}
\item The Colombeau algebra $\cG(\Omega)$: A family $(u_\varepsilon)_{\varepsilon\in(0,1]}$ represents an element of $\cG(\Omega)$, if every derivative
$\partial^\alpha u_\varepsilon$ is $O(\varepsilon^{-b})$ on compact sets, for some $b\geq 0$. The inclusion $\cD'(\Omega)\subset\cG(\Omega)$ holds, constructed by cut-off and convolution with a mollifier. $\mathcal{C}^\infty(\Omega)$ is a faithful subalgebra. $\cG(\Omega)$ is invariant by superposition with maps $f\in\mathcal{O}_M(\R)$.
\item The subalgebra $\cG^0(\Omega)$, characterized by the property that all derivatives $\partial^\alpha u_\varepsilon$ are $O(1)$ on every compact set. It holds that
$\cG^0 \cap \cD'(\Omega) = \cC^\infty(\Omega)$, and $\cG^0(\Omega)$ is invariant under superposition with arbitrary functions $f\in\cC^\infty(\Omega)$.
\end{itemize}
Similar definitions apply to Colombeau generalized functions on the closure of an open set.
The algebra $\cG^0(\Omega)$ has been introduced in the context of nonlinear regularity theory \cite{Delcroix:2007, O:2006} and may be viewed as a subalgebra of \emph{regular generalized functions}.

It has been known for a long time that for globally Lipschitz $f\in\mathcal{O}_M(\R)$, problem \eqref{eq : NLWE} has a unique solution in $\cG([0,T]\times\R^d)$ for arbitrary initial data in $\cG(\R^d)$ and any $T>0$ in space dimensions $d= 1,2, 3$, see e.~g. \cite{Colombeau:1985, O:1987, MORusso:1998}. In one space dimension and for globally Lipschitz $f\in\mathcal{O}_M(\R)$, existence and uniqueness of a solution in $\cG^0([0,T]\times\R)$ has been shown in \cite{O:2006}. For power nonlinearities, existence and uniqueness results in an $L^2$-based Colombeau algebra have been obtained in space dimensions $d\leq 9$ with suitable bounds on the polynomial growth of $f$ in \cite{NOP:2005}.

The main result of the paper is Theorem\;\ref{thm : existence}. It says that equation\;\eqref{eq : NLWE} has unique solutions in $\cG^0([0,T]\times\R^d)$ for arbitrary initial data in $\cG^0(\R^d)$ and any $T>0$ in space dimensions $d= 1,2, 3$ without any growth or sign restrictions on $f$, provided $h(\varepsilon) = O(\varepsilon^b)$ for some $b>0$.
As mentioned, the proof is based on a contraction mapping argument in the sharp topology on $\cG^0([0,T]\times\R^d)$. A map on this space is a contraction if it is regularizing; this makes it possible to have solutions for arbitrary $T$ and arbitrary smooth $f$.

One may see the result also as a classical global existence result for bounded data and small nonlinearities. In fact, Theorem\;\ref{thm : existence} delivers a net $(u_\varepsilon)_{\varepsilon\in(0,1]}$ of smooth functions that satisfies equation \eqref{eq : NLWE} up to a term of $O(\eps^\infty)$. Further, by uniqueness of the solution in the Colombeau algebra, the difference of any two such nets is also $O(\eps^\infty)$ (see Corollary\;\ref{prop : veryweak}). In the classical literature, such solutions have been considered as \emph{asymptotic solutions} \cite{MC:1979}, \emph{semiclassical solutions} \cite{Robert:1987} or \emph{very weak solutions} \cite{GR:2015, RT:2017, RT:2017a}.
A novelty of the paper appears to be that the generalized result in the Colombeau algebra is proved first, using the sharp topology and the contraction mapping principle on  $\cG^0([0,T]\times\R^d)$, and the classical result is deduced from it (Corollary\;\ref{prop : veryweak}). One could also go the other way round, but this appears to be much more tedious (Remark\;\ref{rem : classical}). The reason for the existence of solutions of arbitrary lifespan without growth and sign conditions on the nonlinearity is further illuminated in Example\;\ref{ex : lifespan}.

Finally, we show that the generalized solution with initial data in $\cG^0(\R^d)$ is associated with the corresponding solution of the linear equation (with $f\equiv 0$), that is, the difference of the defining representatives converge weakly to zero (Proposition\;\ref{prop : association}).

The paper ends with an appendix on $L^\infty$-estimates in the linear wave equation in space dimensions $d=1,2,3$. These estimates are known, but we found it convenient to collect them in a form needed in our arguments.

We mention that a fixed point principle of Leray-Schauder type in the Colombeau framework has been established and employed in the elliptic setting in \cite{PilipovicScarpalezos:2006}, working directly on the space of representing families. In the framework of $(\cC,\cE,\cP)$-algebras, a fixed point theorem has been formulated and applied in \cite{Marti:2015}, also using estimates on the representing families. Further, in the framework of \emph{generalized smooth functions} \cite{GKV:2015}, general fixed point principles have been established and used to solve ODEs and PDEs with nonlinearities given by generalized functions, see \cite{BG:2019, BG:2019a} and the remarks in \cite{LBG:2019}.

\section{Colombeau algebras}\label{sec : 2}

We will employ the {\it special Colombeau algebra} of generalized functions denoted by $\G^{s}$ in \cite{GKOS:2001} (called the {\it simplified Colombeau algebra} in \cite{B:1990}). However, here we will simply use the letter $\G$ instead.
This section serves to recall the definitions and properties required for our purpose. For more details, see e.~g. \cite{Colombeau:1984, Colombeau:1985, GKOS:2001, NPS:1998, O:1992}.

Given a non-empty open subset $\Omega$ of $\mathbb{R}^n$, the space of real valued, infinitely differentiable functions on $\Omega$ is denoted by $\Cinf(\Omega)$, while $\Cinf(\overline{\Omega})$ refers to the subspace of functions all whose derivatives have a continuous extension up to the closure of $\Omega$.

Let $\Cinf(\Omega)^{(0,1]}$ be the differential algebra of all maps from the interval $(0,1]$ into $\Cinf(\Omega)$. Thus each element of $\Cinf(\Omega)^{(0,1]}$ is a family $(u_{\varepsilon})_{\varepsilon \in (0,1]}$ of real valued smooth functions on $\Omega$. The subalgebra $\EM(\Omega)$ of \emph{moderate} nets is defined by the elements
$(u_{\varepsilon})_{\varepsilon \in (0,1]}$ of $\Cinf(\Omega)^{(0,1]}$ with the property that, for all $K \Subset \Omega$ and $\alpha \in \mathbb{N}_0^n$, there exists $b\in\R$ such that
\begin{equation}\label{eq:moderate}
	\sup_{x \in K} |\partial^{\alpha} u_{\varepsilon}(x)| = O(\varepsilon^{b}) \quad {\rm as}\ \varepsilon \to 0.
\end{equation}
The ideal $\Neg(\Omega)$ of \emph{negligible} nets is defined by all elements $(u_{\varepsilon})_{\varepsilon \in (0,1]}$ of $\Cinf(\Omega)^{(0,1]}$ with the property that, for all $K \Subset \Omega$, $\alpha \in\mathbb{N}_0^n$ and $a \ge 0$,
\begin{equation}\label{eq:negligible}
	\sup_{x \in K} |\partial^{\alpha} u_{\varepsilon}(x)| = O(\varepsilon^a) \quad {\rm as}\ \varepsilon \to 0.
\end{equation}
The {\it Colombeau algebra} $\G(\Omega)$ {\it of generalized functions} is defined as the factor space
\[
	\G(\Omega) = \EM(\Omega) / \Neg(\Omega).
\]
The Colombeau algebra $\G(\overline{\Omega})$ on the closure of $\Omega$ is constructed in a similar way: the compact subsets $K$ occurring in the definition are now compact subsets of $\overline{\Omega}$, i.e., may reach up to the boundary. Since $\EM(\overline{\Omega}) \subset \EM(\Omega)$ and $\Neg(\overline{\Omega}) \subset \Neg(\Omega)$, there is a canonical map $\G(\overline{\Omega}) \to \G(\Omega)$. However, this map is not injective, as follows from the fact that $\Neg(\Omega)\cap \EM(\overline{\Omega}) \ne \Neg(\overline{\Omega})$.

\emph{Restrictions to open subsets.}
Let $\omega$ be an open subset of $\Omega$ and $U\in\cG(\Omega)$. Then the restriction $U|\omega$, obtained by restriction of representatives, is a well defined element of $\cG(\omega)$. The \emph{support} of a generalized function $U\in\cG(\Omega)$, denoted by $\supp U$, is the complement of the largest open set $\omega\subset\Omega$ such that
$U|\omega = 0$. An analogous definition applies to elements $U\in\G(\overline{\Omega})$ with the sets $\omega$, on which $U$ vanishes, taken as open in $\overline{\Omega}$.
Similarly, the restriction of $U\in\cG(\Omega)$ or $U\in\G(\overline{\Omega})$ to lower dimensional linear subspaces can be defined. This gives a meaning to the initial values of elements of $\cG([0,T]\times \R)$ at $t=0$.

\emph{The ring of generalized numbers.} We let $\EM$ be the space of nets $(r_\varepsilon)_{\varepsilon\in(0,1]}$ of real numbers such that $|r_\eps| = O(\varepsilon^b)$ as
$\varepsilon \to 0$ for some $b\in\R$. Similarly, $\Neg$ comprises those sequences which are $O(\varepsilon^a)$ as
$\varepsilon \to 0$ for every $a\geq 0$. The factor space $\widetilde{\R} = \EM/\Neg$ is the \emph{Colombeau ring of generalized numbers}.

\emph{Generalized functions of bounded type.}
Let $\Omega$ be an open subset of $\R^n$ and $L$ a subset of $\Omega$. A generalized function $U$ from $\cG(\Omega)$ is called of
\emph{bounded type on $L$}, if it has a representative $(u_{\varepsilon})_{\varepsilon \in (0,1]}$ such that
\begin{equation}\label{eq:boundedtype}
   \sup_{x\in L}|u_\eps(x)| = O(1) \quad {\rm as}\ \varepsilon \to 0.
\end{equation}
The subalgebra $\cG^0(\Omega)$ comprises the generalized functions from $\cG(\Omega)$ all whose derivatives of any order are of bounded type on compact sets. The bounded type property is defined similarly for generalized functions in $\cG(\overline{\Omega})$, as is the subalgebra $\cG^0(\overline{\Omega})$.

While $\cG(\Omega)$ is invariant under smooth maps of polynomial growth, it is important to note that $\cG^0(\Omega)$ is invariant under arbitrary smooth maps:
\begin{remark}
(a) If $f\in \mathcal{O}_M(\mathbb{R})$ and $U\in \cG(\Omega)$, then $f(U)$ is a well defined element of $\cG(\Omega)$. That is, if $(u_\varepsilon)_{\varepsilon\in(0.1]}$ is a representative of $U$, then $(f(u_\varepsilon))_{\varepsilon\in(0.1]}$ belongs to $\EM(\Omega)$ and its class in $\cG(\Omega)$ does not depend on the choice of representative of $U$.\\
(b) If $f\in \mathcal{C}^\infty(\mathbb{R})$ and $U\in \cG^0(\Omega)$, then $f(U)$ is a well defined element of $\cG^0(\Omega)$.\\
The first assertion can be found already in \cite{Colombeau:1984,Colombeau:1985}; the second assertion is proved e.g. in \cite{O:2006}.
\end{remark}

\emph{The sharp topology.} The sharp topology, introduced in \cite{B:1990}, can be defined through its family of neighborhoods $V(K,p,q)$, where $K$ is a compact subset of $\Omega$, $p \in \N$ and $q \geq 0$. An element $U$ of $\G(\Omega)$
belongs to $\cV(K,p,q)$ if it has a representative $(u_{\varepsilon})_{\varepsilon \in (0,1]}$ such that
\begin{equation}\label{eq:nbhd}
   \sup_{x \in K} \max_{|\alpha|\leq p} |\partial^{\alpha} u_{\varepsilon}(x)| = O(\varepsilon^q) \quad {\rm as}\ \varepsilon \to 0.
\end{equation}
\begin{lemma}\label{lemma : closed}
$\cG^0(\Omega)$ is a closed subspace of $\cG(\Omega)$.
\end{lemma}
\begin{proof}
If $U\not\in \cG^0(\Omega)$, there is a representative $(u_{\varepsilon})_{\varepsilon \in (0,1]}$, a compact subset $K$ of $\Omega$ and a multi-index $\alpha$ such that
\[
   \limsup_{\eps\to 0} \sup_{x\in K} |\partial^\alpha u_\eps(x)| = \infty.
\]
But then the neighborhood $U + \cV(K,|\alpha|,0)$ does not intersect $\cG^0(\Omega)$.
\end{proof}

Let $K$ be a compact subset of $\Omega$. We denote by $\cG_K(\Omega)$ the subspace of elements of $\cG(\Omega)$ whose support is contained in $K$.
\begin{lemma}\label{lem:supports}
$\cG_K(\Omega)$ is a closed subset of $\cG(\Omega)$.
\end{lemma}
\begin{proof}
Suppose that $U$ does not belong to $\cG_K(\Omega)$ and let $(u_\varepsilon)_{\varepsilon\in(0,1]}$ be a representative of $U$. Then there is a compact subset $L$ of the complement of $K$ in $\Omega$, a multi-index $\alpha$ and $q\geq 0$ such that $u_\varepsilon$ is not $O(\varepsilon^q)$ on $L$. This implies that there is a subsequence
$\varepsilon_k \to 0$ such that
\[
   \sup_{x\in L}|\partial^\alpha u_\varepsilon(x)| \geq \varepsilon_k^q.
\]
If $V-U \in \cV(L,|\alpha|,2q)$, then $\sup_{x\in L}|\partial^\alpha v_\varepsilon(x)| \geq \frac12 \varepsilon_k^q$ for sufficiently large $k$, and so the support of $V$ has a non-empty intersection with $L$. Thus the neighborhood $U + \cV(L,|\alpha|,2q)$ of $U$ does not intersect $\cG_K(\Omega)$, showing that $\cG_K(\Omega)$ is closed.
\end{proof}
\emph{Internal sets.} Given a net of subsets $(A_\varepsilon)_{\varepsilon\in(0,1]}$ of $\cG(\Omega)$, we may consider the set $A$ of all $U\in\cG(\Omega)$ having
a representative $(u_\varepsilon)_{\varepsilon\in(0,1]}$ such that $u_\varepsilon\in A_\varepsilon$ for sufficiently small $\varepsilon$. Sets $A$ of this form are called
\emph{internal subsets} of $\cG(\Omega)$.
\begin{remark}\label{rem:internal}
Internal subsets of $\cG(\Omega)$ are closed with respect to the sharp topology. This has been proven in \cite{MOVernaeve:2008}.
\end{remark}
We are going to recall that the sharp topology on $\cG(\Omega)$ can be defined in terms of an ultra-metric. This construction is due to \cite{Scarpalezos:1998, Scarpalezos:2000} and has been further developed by \cite{Garetto:2005}.
Let $(K_n)_{n\in\N}$ be an exhausting sequence of compact subsets of $\Omega$. The seminorms $\mu_n$ on $\cC^\infty(\Omega)$ are given by
\begin{equation}\label{eq : seminorms}
   \mu_n(f) = \sup_{x\in K_n} \sup_{|\alpha|\leq n}|\partial^\alpha f(x)|.
\end{equation}
Valuations $\nu_n : \EM(\Omega)\to(-\infty,\infty]$ can be defined by
\[
   \nu_n\big((u_{\varepsilon})_{\varepsilon \in (0,1]}\big) = \sup_{b\in\R} \{\mu_n(u_\eps) = O(\eps^b) \ {\rm as}\ \varepsilon \to 0\}.
\]
Obviously, $(u_{\varepsilon})_{\varepsilon \in (0,1]}$ belongs to $\Neg(\Omega)$ if and only if $\nu_n\big((u_{\varepsilon})_{\varepsilon \in (0,1]}\big) = \infty$ for all $n$. Thus the valuations can be extended to the factor algebra
$\cG(\Omega)$. The following properties hold:
\begin{itemize}
\item[(a)] $\nu_n(U+V) \geq \min \big(\nu_n(U),\nu_n(V)\big)$;

\item[(b)] $\nu_n(UV) \geq \nu_n(U) + \nu_n(V)$;

\item[(c)] if $m\geq n$, then $\nu_m(U) \leq \nu_n(U)$.
\end{itemize}
By means of these valuations, a family of ultra-pseudo-seminorms on $\cG(\Omega)$ can be defined by
\[
    p_n(U) = \exp(-\nu_n(U)).
\]
They have the properties
\begin{itemize}
\item[(a)] $p_n(U+V) \leq \max \big(p_n(U),p_n(V)\big)$;

\item[(b)]  $p_n(UV) \leq p_n(U)p_n(V)$;

\item[(c)] $p_n(\lambda U) = p_n(U)$ for all $\lambda \in \R$.
\end{itemize}
Finally, an ultra-metric can be defined on $\cG(\Omega)$ by
\[
   d(U,V) = \sum_{n=0}^\infty 2^{-n-1}\min\big(p_n(U-V),1\big).
\]
It is an easy exercise to show that the topology induced on $\cG(\Omega)$ by the ultra-metric $d$ is the same as the one given by the neighborhoods \eqref{eq:nbhd}.
Further, it is known \cite{ Garetto:2005, NPS:1998, Scarpalezos:1998} that $\cG(\Omega)$ with the uniform structure induced by $d$ is complete, hence a complete ultra-metric space.
By Lemma \ref{lemma : closed} the same is true of $\cG^0(\Omega)$.

\emph{Contractions in $\cG^0(\Omega)$.} First, observe that if $U\in \cG^0(\Omega)$, then $\nu_n(U) \geq 0$ for all $n$, hence
$p_n(U)\leq 1$ for all $n$, and
\begin{equation}\label{eq:contrationG0}
   d(U,V) = \sum_{n=0}^\infty 2^{-n-1}p_n(U-V).
\end{equation}
In particular, $\cG^0(\Omega)$ is contained in the unit ball around zero in $\cG(\Omega)$.

Let $\mathcal{M}$ be a subset of $\cG(\Omega)$. As usual, a map $F: \mathcal{M}\to\mathcal{M}$ is called a \emph{contraction}, if there is $\kappa < 1$ such that
\begin{equation*}%\label{eq:contraction}
   d(F(U),F(V)) \leq \kappa d(U,V)
\end{equation*}
for all $U,V\in \mathcal{M}$.

It will be useful to elaborate on sufficient conditions for $F$ to be a contraction on subsets $\mathcal{M}$ of $\cG^0(\Omega)$. First, if there is $\kappa < 1$ such that
\begin{equation}\label{eq:contraction_pn}
    p_n(F(U)-F(V)) \leq \kappa p_n(U-V)
\end{equation}
for all $U,V\in \mathcal{M}$ and $n\in\N$, then $F$ is a contraction, as follows from the formula \eqref{eq:contrationG0}, valid for $U,V\in \cG^0(\Omega)$.
This in turn is equivalent with
\begin{equation*}%\label{eq:contraction_valuation}
  \nu_n\big(F(U) - F(V)\big) \geq -\log\kappa + \nu_n(U-V)
\end{equation*}
for all $U,V\in \mathcal{M}$ and $n\in\N$. Since $-\log\kappa > 0$, this means -- roughly speaking -- that $F$ raises the regularity.
\begin{proposition}\label{prop:contraction}
Let $\mathcal{M}\subset \cG^0(\Omega)$ and $F:\mathcal{M}\to\mathcal{M}$. Assume that $F$ is of the form $F(U) = EG(U)$, where $E$ is a generalized real number such that $\nu_0(E) > 0$  and $G$ is Lipschitz continuous with respect to all $\mu_n$, that is, there are constants $C_n \geq 0$ such that
if $U,V\in \cG(\Omega)$ and $(u_{\varepsilon})_{\varepsilon \in (0,1]}$, $(v_{\varepsilon})_{\varepsilon \in (0,1]}$ are representatives of $U$ and $V$, there is
$\eps_n > 0$ such that
\begin{equation}\label{eq:Lipschitz_n}
   \mu_n(G(u_\eps) - G(v_\eps)) \leq C_n \mu_n(u_\eps - v_\eps)
\end{equation}
for all $\eps\leq \eps_n$. Then $F$ is a contraction on $\mathcal{M}$.
\end{proposition}
\begin{proof}
We begin by observing that if $\mu_n(u_\eps - v_\eps) = O(\eps^b)$, then also $C_n\mu_n(u_\eps - v_\eps) = O(\eps^b)$. Inequality \eqref{eq:Lipschitz_n} implies that
$\mu_n(G(u_\eps) - G(v_\eps)) = O(\eps^b)$ as well. It follows that
\[
\{b\in\R:\mu_n(u_\eps - v_\eps) = O(\eps^b)\} \subset \{b\in\R: \mu_n(G(u_\eps) - G(v_\eps)) = O(\eps^b)\}.
\]
Therefore, the supremum of the left-hand side is less or equal to the supremum of the right-hand side. This means that $\nu_n(U-V) \leq \nu_n\big(G(U) - G(V)\big)$, which in turn implies
\[
  \nu_n\big(F(U) - F(V)\big) \geq \nu_n(E) + \nu_n\big(G(U) - G(V)\big) \geq \nu_0(E) + \nu_n(U - V)
\]
using that for generalized real numbers $\nu_n(E) = \nu_0(E)$ for all $n$.
This in turn yields that
\[
   p_n\big(F(U) - F(V)\big) \leq \exp(-\nu_0(E))p_n(U-V),
\]
thus \eqref{eq:contraction_pn} holds with $\kappa = \exp(-\nu_0(E)) < 1$.
\end{proof}

\section{Global existence of regular solutions}
\label{sec : 3}

This section is devoted to establishing existence and uniqueness of a solution in $\mathcal{G}^0([0,T]\times \mathbb{R}^d)$ to the Cauchy problem for nonlinear wave equations with small nonlinearity in space dimensions $d = 1,2,3$. We consider the problem
\begin{equation}\label{eq : regular nonlinear wave equation}
\begin{array}{lr}
\partial_t^2U - \Delta U = Ef(U) & \mbox{in}\ \G^0([0,T]\times\mathbb{R}^d),\vspace{4pt} \\
U|_{t = 0} = U_0,\quad \partial_tU|_{t = 0} = U_1 & \mbox{in}\ \G^0(\mathbb{R}^d). \vspace{4pt} \\
\end{array}
\end{equation}
We make the following assumptions.
\begin{itemize}
\item[(A1)]
$E$ is a generalized number with positive valuation $\nu_0(E) > 0$;
\item[(A2)] $f$ belongs to $\mathcal{C}^\infty(\mathbb{R})$ and satisfies $f(0) = 0$;
\item[(A3)] $U_0$, $U_1$ belong to $\mathcal{G}^0(\mathbb{R}^d)$ and are compactly supported, with support in the ball
$B_0 = \{x\in\mathbb{R}^d: |x|\leq r\}$ of radius $r$, for some $r\geq 0$.
\end{itemize}
Note that the generalized number $E$ with representative $(e_\varepsilon)_{\eps\in(0,1]}$ takes the role of the small parameter $h(\varepsilon)$ informally introduced in the introduction.

We set
\[
   K_0 = \{(t,x)\in[0,T]\times\mathbb{R}^d:|x|\leq t + r\}.
\]
The (classical) solution operator for the linear wave equation with initial data $u_0, u_1$ and right-hand side $h$ is denoted by
$ L(u_0,u_1,h)$ (see appendix).
\begin{theorem}\label{thm : existence}
Let $d = 1,2,3$ and let the assumptions $(A1),$ $(A2)$ and $(A3)$ hold. Then for any $T > 0,$ problem \eqref{eq : regular nonlinear wave equation}
has a unique solution $U \in \mathcal{G}^0([0,T] \times \mathbb{R}^d)$ with support in $K_0$.
\end{theorem}

\begin{proof}
We define the map $F:\mathcal{G}^0([0,T] \times \mathbb{R}^d)\to \mathcal{G}([0,T] \times \mathbb{R}^d)$ by
\begin{equation*}%\label{eq : map}
    F(U) = L(U_0,U_1,Ef(U))\ =\ L(U_0,U_1,0) + EG(U)
\end{equation*}
with
\[
   G(U) = L(0,0,f(U)).
\]
We wish to show, using Proposition\;\ref{prop:contraction}, that $F$ is a contraction on a suitable closed subset $\mathcal{M}$ of $\mathcal{G}^0([0,T] \times \mathbb{R}^d)$.
As in the appendix, we take $K_0$ as a starting element of an exhausting sequence $(K_n)_{n\in\N}$ of compact subsets of $[0,T]\times \mathbb{R}^d$.
The seminorms $\mu_n$ on $\mathcal{C}^\infty([0,T]\times \mathbb{R}^d)$ are defined as in \eqref{eq : seminorms}.
Further, the seminorms $\mu_n^0$ on $\mathcal{C}^\infty(\mathbb{R}^d)$ are defined in an analogous way (see appendix).
In what follows, $(u_{0\varepsilon})_{\varepsilon\in(0,1]}$ and $(u_{1\varepsilon})_{\varepsilon\in(0,1]}$ denote representatives of $U_0$ and $U_1$, respectively.

We let $\mathcal{M}_1$ comprise the elements $U$ of $\mathcal{G}([0,T] \times \mathbb{R}^d)$ having a representative $(u_\varepsilon)_{\varepsilon\in(0,1]}$
with the property
\begin{itemize}
   \item[] For all $n\in\mathbb{N}$ there is $\varepsilon_n>0$ such that
    $\mu_n\big(u_\varepsilon - L(u_{0\varepsilon},u_{1\varepsilon},0)\big) \leq 1$\\
     for $0< \varepsilon \leq \varepsilon_n$.
\end{itemize}
Clearly, $\mathcal{M}_1$ is an internal subset of $\mathcal{G}([0,T] \times \mathbb{R}^d)$, hence closed with respect to the sharp topology (Remark\;\ref{rem:internal}). Next, $\mathcal{M}_2$ denotes the set of elements of $\mathcal{G}([0,T] \times \mathbb{R}^d)$ with support contained in $K_0$. By Lemma\;\ref{lem:supports}, $\mathcal{M}_2$ is a closed subset of
$\mathcal{G}([0,T] \times \mathbb{R}^d)$ as well. Finally, we set
\[
   \mathcal{M} = \mathcal{M}_1 \cap \mathcal{M}_2 \cap \mathcal{G}^0([0,T] \times \mathbb{R}^d).
\]
By Lemma\;\ref{lemma : closed}, $\mathcal{M}$ is a closed subset of $\mathcal{G}([0,T] \times \mathbb{R}^d)$, hence a complete ultra-metric space.

We first show that $F$ maps $\mathcal{M}$ into itself. Let $U \in \mathcal{M}$ and take a representative $(u_\varepsilon)_{\varepsilon\in (0,1]}$. Using that
$\mathcal{G}^0([0,T] \times \mathbb{R}^d)$ is invariant under superposition with smooth maps, it follows that $f(U)$ belongs to $\mathcal{G}^0([0,T] \times \mathbb{R}^d)$ as well. This means, in particular, that for all $n\in\N$ there are constants $C_n'$ and $\varepsilon_n'$ such that
\[
   \mu_n\big(f(u_\varepsilon)\big) \leq C_n'\quad \mbox{for}\quad 0<\varepsilon\leq\varepsilon_n'.
\]
A similar assertion holds for $\mu_n^0(u_{0\varepsilon})$ and $\mu_n^0(u_{1\varepsilon})$. Recall that $\nu_0(E) > 0$ means that there is $b>0$ such that
\begin{equation}\label{eq : E}
   |e_\varepsilon| = O(\varepsilon^b) \quad \mbox{as}\quad \varepsilon\to 0,
\end{equation}
where $(e_\varepsilon)_{\varepsilon\in(0,1]}$ is a representative of $E$.
Using this and Lemma\;\ref{lem:waveestimates}(b), we conclude that $F(U) = L(U_0,U_1,0) + EL(0,0,f(U))$ belongs to $\mathcal{G}^0([0,T] \times \mathbb{R}^d)$
as well. Also by Lemma\;\ref{lem:waveestimates}(a), $F(U)$ has its support in $K_0$, thus belongs to $\mathcal{M}_2$.

To show that $F(U) \in \mathcal{M}_1$, we have to estimate
\[
   \mu_n\big(L(u_{0\varepsilon}, u_{1\varepsilon}, e_\varepsilon f(u_\varepsilon)) - L(u_{0\varepsilon}, u_{1\varepsilon}, 0)\big)
       = |e_\varepsilon| \mu_n\big(L(0, 0, f(u_\varepsilon))\big).
\]
But the right-hand side will eventually be less or equal to 1, due to \eqref{eq : E} and the fact that $L(0, 0, f(U))$ belongs to $\mathcal{G}^0([0,T] \times \mathbb{R}^d)$.
Combining the assertions, we see that $F(U)$ belongs to $\mathcal{M}$.

Second, in order to show that $F$ is a contraction on $\mathcal{M}$, we invoke Proposition\;\ref{prop:contraction} and just have to establish the Lipschitz estimates \eqref{eq:Lipschitz_n} for
the map $U\to G(U) = L(0,0,f(U))$. Let $U,V$ belong to $\mathcal{M}$ with representatives $(u_\varepsilon)_{\varepsilon\in(0,1]}$, $(v_\varepsilon)_{\varepsilon\in(0,1]}$.
We have to estimate $\mu_n\big(L(0,0,f(u_\varepsilon) - f(v_\varepsilon))\big)$.

For $n=0$ this is fairly straightforward. In fact,
\[
   f(u_\varepsilon) - f(v_\varepsilon) = (u_\varepsilon - v_\varepsilon)\int_0^1 f'(v_\varepsilon + \tau (u_\varepsilon - v_\varepsilon))d\tau.
\]
But $U$ belongs to $\mathcal{M}_1$, so
\[
   \sup_{(t,x)\in K_0}|u_\varepsilon(t,x)| \leq \mu_0\big(L(u_{0\varepsilon},u_{1\varepsilon},0)\big) + 1
      \leq C_0\big(\mu_1^0(u_{0\varepsilon}) + \mu_0^0(u_{1\varepsilon})\big) + 1
\]
by Lemma\;\ref{lem:waveestimates}. Since $U_0$ and $U_1$ belong to $\mathcal{G}^0(\mathbb{R}^d)$, the right-hand side is bounded by a constant. The same argument applies to $v_\varepsilon$.
But $f'$ is bounded on bounded sets, so there is a constant $C>0$ such that
\[
   \mu_0\big(f'(u_\varepsilon + \tau (u_\varepsilon - v_\varepsilon))\big) \leq C
\]
and so
\[
  \mu_0\big(f(u_\varepsilon) - f(v_\varepsilon)\big) \leq C\mu_0\big(u_\varepsilon - v_\varepsilon\big).
\]
From Lemma\;\ref{lem:waveestimates},
\[
  \mu_0\big(L(0,0,f(u_\varepsilon) - f(v_\varepsilon))\big) \leq CC_0\mu_0\big(u_\varepsilon - v_\varepsilon\big)
\]
and this is the desired estimate of order zero. To indicate the estimate of order one, denote by $\partial_x$ a spatial derivative. Using the explicit formulas in
the appendix, it holds that
$\partial_x L(0,0,f(u_\varepsilon) - f(v_\varepsilon)) =  L(0,0,\partial_x(f(u_\varepsilon) - f(v_\varepsilon)))$ and
\begin{eqnarray*}
  \partial_x(f(u_\varepsilon) - f(v_\varepsilon)) & = & f'(u_\varepsilon)(\partial_x u_\varepsilon - \partial_x v_\varepsilon)\\
      && +\ (u_\varepsilon - v_\varepsilon)\,\partial_x v_\varepsilon\int_0^1 f''(v_\varepsilon + \tau (u_\varepsilon - v_\varepsilon))d\tau.
\end{eqnarray*}
One may now apply the same argument as for the zero order estimate, observing that $\partial_x v_\varepsilon$ is bounded on $K_0$ since $V$ belongs to $\mathcal{M}$.
The time derivatives are yet more complicated, because $\partial_t L(0,0,f(u_\varepsilon))$ involves also the evaluation of $f(u_\varepsilon)$ at $t=0$. Nevertheless, the argument can be continued recursively, yielding the estimates \eqref{eq:Lipschitz_n} for every $n$.

We conclude that there is a unique $U\in\cM$ such that $U = F(U)$. Written out in terms of representatives, this means that
\[
   u_\eps = L\big(u_{0\varepsilon}, u_{1\varepsilon}, e_\varepsilon f(u_\eps)\big) + n_\eps,
\]
where the net $(n_\eps)_{\eps\in(0,1]}$ belongs to $\cN([0,T] \times \mathbb{R}^d)$. It follows from the formulas in the appendix that
\begin{equation}\label{eq : solution}
    \partial_t^2u_\eps - \Delta u_\eps = e_\varepsilon f(u_\eps) + m_\eps,
\end{equation}
where $m_\eps = (\p_t^2 - \Delta)n_\eps$ represents an element of $\cN([0,T] \times \mathbb{R}^d)$ as well. It also holds that
$u_\eps|_{t = 0} = u_{0\varepsilon}$ and $\partial_tu_\eps|_{t = 0} = u_{1\varepsilon}$, thus $U$ is the unique solution to problem\;\eqref{eq : regular nonlinear wave equation}
in $\mathcal{M}$.

It remains to show that the solution is unique in $\mathcal{G}^0([0,T] \times \mathbb{R}^d)$ with support in $K_0$. Thus let $W$ be a solution
of \eqref{eq : regular nonlinear wave equation} with the indicated properties. This means that for any representative $(w_\varepsilon)_{\varepsilon\in(0,1]}$, there are elements $(m_\varepsilon)_{\varepsilon\in(0,1]}\in \mathcal{N}([0,T] \times \mathbb{R}^d)$
and $(n_{0\varepsilon})_{\varepsilon\in(0,1]}, (n_{1\varepsilon})_{\varepsilon\in(0,1]}\in \mathcal{N}(\mathbb{R}^d)$ such that
\begin{equation}\label{eq : waverepresentative}
\begin{array}{lr}
\partial_t^2 w_\varepsilon - \Delta w_\varepsilon = e_\varepsilon f(w_\varepsilon) + m_\varepsilon, \\
w_\varepsilon|_{t = 0} = u_{0\varepsilon} + n_{0\varepsilon},\quad \partial_t w_\varepsilon|_{t = 0} = u_{1\varepsilon} + n_{1\varepsilon}.
\end{array}
\end{equation}
Since $w_\varepsilon$ is a classical solution of \eqref{eq : waverepresentative}, it follows that
\[
   w_\varepsilon = L(u_{0\varepsilon} + n_{0\varepsilon}, u_{1\varepsilon} + n_{1\varepsilon},e_\varepsilon f(w_\varepsilon) + m_\varepsilon).
\]
Using the linearity of the solution operator $L$, we conclude that
\begin{equation}\label{eq : Mestimate}
  \mu_n\big(w_\varepsilon - L(u_{0\varepsilon}, u_{1\varepsilon},0)\big)
      \leq |e_\varepsilon| \mu_n\big(L(0,0,f(w_\varepsilon))\big) + \mu_n\big(L(n_{0\varepsilon}, n_{1\varepsilon}, m_\varepsilon)\big).
\end{equation}
Since $W$ belongs to $\mathcal{G}^0([0,T] \times \mathbb{R}^d)$, all seminorms $\mu_n(w_\varepsilon)$ are bounded as $\varepsilon\to 0$. From Lemma\;\ref{lem:waveestimates}, the seminorms $\mu_n\big(L(0,0,f(w_\varepsilon))\big)$ are bounded as well. Also, all seminorms of $n_{0\varepsilon}$, $n_{1\varepsilon}$, $m_\varepsilon$ are of order $O(\varepsilon^q)$ for every $q\geq 0$, so the same is true of
$\mu_n\big(L(n_{0\varepsilon}, n_{1\varepsilon}, m_\varepsilon)\big)$. Using \eqref{eq : E}, it follows that the right-hand side in \eqref{eq : Mestimate} will eventually be less than $1$ as $\varepsilon\to 0$. Thus $W$ is seen to belong to $\mathcal{M}$ and hence coincides with the unique solution in $\mathcal{M}$.
\end{proof}
\begin{remark}\label{rem : uniqueness}
At this stage it is not clear whether the uniqueness assertion of Theorem\;\ref{thm : existence} holds without assuming that the support is contained in $K_0$. However, if $f$ belongs to $\mathcal{O}_M(\mathbb{R})$ and is globally Lipschitz, uniqueness holds even in $\mathcal{G}([0,T] \times \mathbb{R}^d)$
without further assumptions, as was noted in the introduction.
\end{remark}
Theorem\;\ref{thm : existence} can be restated in classical terms. In what follows, the notation $O(\eps^\infty)$ signifies a net of smooth functions vanishing to any order, that is, satisfying the negligibility estimate \eqref{eq:negligible} for every $a\geq 0$.
\begin{corollary}\label{prop : veryweak}
Given nets $(u_{0\varepsilon}, u_{1\varepsilon})_{\varepsilon\in(0,1]}$ of smooth functions of compact support in $B_0$ satisfying the bounded type condition \eqref{eq:boundedtype}, there is a net $(u_\varepsilon)_{\varepsilon\in(0,1]}$ of smooth functions, satisfying the moderateness condition \eqref{eq:moderate}, which solves the initial value problem that
\begin{equation}\label{eq : NLWEclassical}
\begin{array}{l}
	\partial_t^2u_\eps - \Delta u_\eps = e_\varepsilon f(u_\eps) + O(\eps^\infty),  \quad t \in [0,T], \ x \in \mathbb{R}^d, \vspace{4pt}\\
	u_\eps|_{t = 0} = u_{0\varepsilon},\quad \partial_tu_\eps|_{t = 0} = u_{1\varepsilon}, \quad x \in \mathbb{R}^d
\end{array}
\end{equation}
Further, any two nets solving \eqref{eq : NLWEclassical} with supports in $K_0$ differ by $O(\eps^\infty)$ as well.
\end{corollary}
\begin{proof}
This is simply a restatement of Theorem\;\ref{thm : existence} in terms of representatives, invoking the meaning of the solution concept in $\mathcal{G}^0([0,T] \times \mathbb{R}^d)$, in particular, \eqref{eq : solution}.
\end{proof}
We complete the investigation by showing that, for $\mathcal{G}^0$-initial data, the generalized solution given by Theorem\;\ref{thm : existence} is actually associated with the solution to the linear wave equation with the same initial data. Recall that two elements $U$, $V$ of $\cG([0,T]\times\mathbb{R}^d)$ are \emph{associated}, if
\[
   \lim_{\varepsilon\to 0}\int_0^T\int_{\R^d}\big(u_\varepsilon(t,x) - v_\varepsilon(t,x)\big)\psi(t,x)dxdt = 0
\]
for every $\psi\in \mathcal{C}^\infty([0,T]\times\mathbb{R}^d)$ of compact support, where $(u_\varepsilon)_{\varepsilon\in(0,1]}$ and $(v_\varepsilon)_{\varepsilon\in(0,1]}$
are representatives of $U$ and $V$, respectively.
\begin{proposition}\label{prop : association}
Under the assumptions of Theorem\;\ref{thm : existence}, let $U\in\cG^0([0,T]\times\mathbb{R}^d)$ be the unique solution to \eqref{eq : regular nonlinear wave equation}. Let $V\in\cG^0([0,T]\times\mathbb{R}^d)$ be the unique solution to equation \eqref{eq : regular nonlinear wave equation}
with $f\equiv 0$ and the same initial data $U_0, U_1$.  Then $U$ and $V$ are associated.
\end{proposition}
\begin{proof}
Note that the solution $V$ to the linear wave equation exists and is unique according to Theorem\;\ref{thm : existence} and Remark\;\ref{rem : uniqueness}.
The proof of Proposition\;\ref{prop : association} proceeds in the same way as the proof of the uniqueness part of Theorem\;\ref{thm : existence}.
Let $(u_\varepsilon)_{\varepsilon\in(0,1]}$ and $(v_\varepsilon)_{\varepsilon\in(0,1]}$ be representatives of $U$ and $V$.
There are elements $(m_\varepsilon)_{\varepsilon\in(0,1]}\in \mathcal{N}([0,T] \times \mathbb{R}^d)$
and $(n_{0\varepsilon})_{\varepsilon\in(0,1]}, (n_{1\varepsilon})_{\varepsilon\in(0,1]}\in \mathcal{N}(\mathbb{R}^d)$ such that
\begin{equation*}
\begin{array}{lr}
\partial_t^2 (u_\varepsilon - v_\varepsilon) - \Delta (u_\varepsilon - v_\varepsilon) = e_\varepsilon f(u_\varepsilon) + m_\varepsilon, \\
(u_\varepsilon - v_\varepsilon)|_{t = 0} = n_{0\varepsilon},\quad \partial_t (u_\varepsilon - v_\varepsilon)|_{t = 0} = n_{1\varepsilon}.
\end{array}
\end{equation*}
It follows that
\[
   (u_\varepsilon - v_\varepsilon) = L(n_{0\varepsilon}, n_{1\varepsilon},e_\varepsilon f(u_\varepsilon) + m_\varepsilon)
\]
and thus, in particular,
\begin{equation}\label{eq : assocestimate}
  \mu_0(u_\varepsilon - v_\varepsilon)
      \leq |e_\varepsilon| \mu_0\big(L(0,0,f(u_\varepsilon))\big) + \mu_0\big(L(n_{0\varepsilon}, n_{1\varepsilon}, m_\varepsilon)\big).
\end{equation}
Since $U$ belongs to $\mathcal{G}^0([0,T] \times \mathbb{R}^d)$, the seminorms $\mu_0(u_\varepsilon)$ are bounded as $\varepsilon\to 0$. From Lemma\;\ref{lem:waveestimates}, the seminorms $\mu_0\big(L(0,0,f(u_\varepsilon))\big)$ are bounded as well. Also, all seminorms of $n_{0\varepsilon}$, $n_{1\varepsilon}$, $m_\varepsilon$ are of order $O(\varepsilon^q)$ for every $q\geq 0$, so the same is true of
$\mu_0\big(L(n_{0\varepsilon}, n_{1\varepsilon}, m_\varepsilon)\big)$. Using \eqref{eq : E}, it follows that the right-hand side in \eqref{eq : assocestimate} tends to zero as $\varepsilon\to 0$. Thus
\[
\lim_{\varepsilon\to 0}\mu_0(u_\varepsilon - v_\varepsilon) = 0.
\]
In particular, $U$ and $V$ are associated.
\end{proof}
\begin{remark}\label{rem:strongassoc}
The estimate \eqref{eq : E} actually shows that $U$ and $V$ are \emph{strongly associated} in the sense of \cite{NPS:1998}.
\end{remark}
Finally, a few words of explanation appear appropriate, why it is possible to have global solutions (on arbitrary time intervals $[0,T]$) for small nonlinearities without any growth or sign restrictions on the nonlinearity (and large initial data). In fact, this is strongly related to classical global existence results for small initial data (and large nonlinearities).
\begin{example}\label{ex : lifespan}
(a) An illustrative example is given by the ordinary differential equation
\begin{equation}\label{eq:ode}
   y'(t) = \eps y^2(t),\quad y(0) = 1.
\end{equation}
The local solution is
\[
   y(t) = \frac{1}{1-\eps t}.
\]
At fixed $\eps$, the solution exists for $0\leq t <  1/\eps$ and thus has finite lifespan. Viewed differently, at fixed, but arbitrary $T$, the solutions exist for $0\leq\eps <  1/T$. This observation makes it possible to construct a net of solutions $(y_\eps)_{\eps\in(0,1]}$ which solves equation \eqref{eq:ode} asymptotically as $\eps\to 0$ with given lifespan $T$, or alternatively provides a representative of a generalized solution in the Colombeau algebra $\cG[0,T]$.

The relation with small data solutions is immediately established by rescaling. Indeed, $z(t) = \eps y(t)$ solves
\begin{equation*}
   z'(t) = z^2(t),\quad z(0) = \eps
\end{equation*}
in the same range of $\eps$ and $t$.

(b) The example can be easily upgraded to provide a solution to a semilinear wave equation with small nonlinerity (which does not depend on the space variable $x\in\R^d$). Indeed, $u_\eps(t,x) = 1/(1-\eps t)$ is a classical solution to
\begin{equation*}
\begin{array}{lr}
\partial_t^2 u_\varepsilon - \Delta u_\varepsilon = 2\varepsilon^2 u_\varepsilon^3, \\
u_\varepsilon|_{t = 0} = 1,\quad \partial_t u_\varepsilon|_{t = 0} = \varepsilon
\end{array}
\end{equation*}
in any space dimension $d$. The solution exists on the time interval $[0,T]$ provided $0 \leq \eps < 1/T$.

(c) Corollary\;\ref{prop : veryweak} is strongly related to classical results on the life\-span of solutions to semilinear wave equations with small initial data. For example, following \cite{Lindblad:1990}, there is the following lifespan estimate for smooth solutions to
\begin{equation}\label{eq : Lindblad}
\begin{array}{l}
	\partial_t^2v - \Delta v = v^2,  \quad t \in [0,T], \ x \in \mathbb{R}^3, \vspace{4pt}\\
	v|_{t = 0} = \eps \psi_0,\quad \partial_t v|_{t = 0} = \eps\psi_1, \quad x \in \mathbb{R}^3
\end{array}
\end{equation}
where $\psi_0,\psi_1$ are smooth functions of compact support: There is $\mu$ and $\eps_0$ such that the solution to \eqref{eq : Lindblad} exists at least for $T< \mu^2/\eps_0^2$, see \cite[Theorem 2.2]{Lindblad:1990}. Setting $u_\eps = v/\eps$, we obtain a solution to
\begin{equation*}
\begin{array}{l}
	\partial_t^2u_\eps - \Delta u_\eps = \eps u_\eps^2,  \quad t \in [0,T], \ x \in \mathbb{R}^3, \vspace{4pt}\\
	u_\eps|_{t = 0} =  \psi_0,\quad \partial_t u_\eps|_{t = 0} = \psi_1, \quad x \in \mathbb{R}^3,
\end{array}
\end{equation*}
and this solution exists on the time interval $[0,T]$ whenever $\eps < \mu/\sqrt{T}$.
\end{example}
\begin{remark}\label{rem : classical}
It is possible to prove Theorem\;\ref{thm : existence} by classical means. One would start by constructing a fixed point of the integral equation
\[
   u = L\big(u_0,u_1,h(\eps)f(u)\big)
\]
in a ball $\cB$ of radius 1 around $L(u_0,u_1,0)$ in $\cC([0,T]\times\R^d)$. The lifespan $T$ is related to the size of the initial data and bounds on $h(\eps)f$ and its first derivatives on the bounded subset of $[0,T]\times\R^d$ spanned by the range of the initial data (plus 1). As $h(\eps)\to 0$, these bounds become smaller than 1, thus turning $L$ into a contraction on the ball $\cB$ in $\cC([0,T]\times\R^d)$ for $\eps<\eps_0$. Next, one would recursively repeat the argument in $\cC^k([0,T]\times\R^d)$, $k\geq 1$, employing a priori estimates on the previously obtained solution in $\cC^{k-1}([0,T]\times\R^d)$ in order to keep the lifespan $T$ and $\eps_0$ fixed. This results in a net $(u_\eps)_{\eps\in(0,\eps_0)}$ of smooth solutions. Finally, one has to establish the $\cG^0$-estimates for existence of a generalized solution and the $\cN$-estimates for uniqueness. Clearly, this path is more tedious than the one which we used with the contraction mapping argument in the sharp topology.
\end{remark}

\subsection*{Acknowledgments}
The results in this paper were obtained during several visits of the first author to Universit\"{a}t Innsbruck. He expresses his heartfelt thanks to the Unit of Engineering Mathematics for the warm hospitality during his visits.

\appendix

\section{Appendix: Linear estimates}
\label{sec:app}

The appendix collects the required estimates for solutions to the linear wave equation in space dimensions $d = 1,2,3$.
The solution formulas are well known; nevertheless, we found it useful to display them here so that the reader can easily visualize the aguments needed for the proof of Lemma\;\ref{lem:waveestimates} below.
We consider classical, smooth solutions to
\begin{equation}\label{eq : linear wave equation}
\begin{array}{lr}
\partial_t^2 u(t,x) - \Delta u(t,x) = h(t,x), & (t,x)\in [0,T]\times\mathbb{R}^d,\vspace{4pt} \\
u(0,x) = u_0(x),\quad \partial_t u(0,x) = u_1(x), & x \in \mathbb{R}^d\vspace{4pt} \\
\end{array}
\end{equation}
with $u_0, u_1 \in \mathcal{C}^\infty(\mathbb{R}^d)$ and $h \in \mathcal{C}^\infty([0,T]\times\mathbb{R}^d)$. The solution
$u \in \mathcal{C}^\infty([0,T]\times\mathbb{R}^d)$ is given by the solution operator $L(u_0,u_1,h)$, which can be obtained by means of convolution with the fundamental solution to the Cauchy problem and Duhamel's principle (see e.g. \cite{TrevesBasic}). The crucial property that allows one to obtain $L^\infty$-estimates in space dimensions $d= 1,2,3$ is the fact that in these three cases, the fundamental solution is a smooth map of time $t\in[0,\infty)$ with values in the space of integrable measures (in fact, integrable functions for $d = 1,2$).

The solution operator $L(u_0,u_1,h)$ is of the following form.
In space dimension $d=1$, $L$ is given by d'Alembert's formula
\begin{eqnarray*}
   L(u_0,u_1,h)(t,x) &=& \frac12(u_0(x+t) + u_0(x-t)) + \frac12\int_{x-t}^{x+t}u_1(y) dy\\
      && +\ \frac12\int_0^t\int_{x-t+s}^{x+t-s}h(s,y) dy ds
\end{eqnarray*}
or alternatively
\begin{eqnarray*}
      &=& \frac12(u_0(x+t) + u_0(x-t)) + \frac12\int_{-t}^{t}u_1(x-y) dy\\
      && + \ \frac12\int_0^t\int_{-s}^{s}h(t-s,x-y) dy ds.
\end{eqnarray*}
In space dimension $d=2$, $L$ is given by Poisson's formula. Namely, $2\pi$ times $L(u_0,u_1,h)(t,x)$ equals
\begin{eqnarray*}
   && \frac{d}{dt}\iint_{|x-y| \leq t} \frac1{\sqrt{t^2-|x-y|^2}}\,u_0(y) dy + \iint_{|x-y| \leq t} \frac1{\sqrt{t^2-|x-y|^2}}\,u_1(y) dy\\
      && +\ \int_0^t\iint_{|x-y| \leq t-s} \frac1{\sqrt{(t-s)^2-|x-y|^2}}\,h(s,y)dy ds
\end{eqnarray*}
or alternatively
\begin{eqnarray*}
    &=& \frac{d}{dt}\iint_{|y| \leq 1} \frac{t}{\sqrt{1-|y|^2}}\,u_0(x-ty) dy + \iint_{|y| \leq 1} \frac{t}{\sqrt{1-|y|^2}}\,u_1(x-ty) dy\\
      && +\ \int_0^t\iint_{|y| \leq 1} \frac{s}{\sqrt{1-|y|^2}}\,h(t-s,x-sy) dy ds.
\end{eqnarray*}
In space dimension $d=3$, $L$ is given by Kirchhoff's formula. Namely, $4\pi$ times $L(u_0,u_1,h)(t,x)$ equals
\begin{eqnarray*}
    && \frac{d}{dt}\iint_{|x-y| = t} \frac1{t}u_0(y) d\omega(y) + \iint_{|x-y| = t} \frac1{t}u_1(y) d\omega(y)\\
      && +\ \int_0^t\iint_{|x-y| = t-s} \frac1{t-s}h(s, y) d\omega(y) ds
\end{eqnarray*}
where $d\omega$ denotes the surface element, or alternatively
\begin{eqnarray*}
    &=& \frac{d}{dt}\iint_{|\omega| = 1} t\, u_0(x-t\omega) d\omega + \iint_{|\omega| = 1} t\, u_1(x - t\omega) d\omega\\
      && +\ \int_0^t\iint_{|\omega| = 1} s\, h(t - s, x - s\omega) d\omega ds.
\end{eqnarray*}
Let $r\geq 0$. We denote the ball of radius $r$ around zero by $B_0 = \{x\in\mathbb{R}^d: |x|\leq r\}$. Further, for $T>0$, we let
\[
   K_0 = \{(t,x)\in[0,T]\times\mathbb{R}^d:|x|\leq t + r\}
\]
and take $K_0$ as a starting element of an exhausting sequence $(K_n)_{n\in\N}$ of compact subsets of $[0,T]\times \mathbb{R}^d$.
The seminorms $\mu_n$ on $\mathcal{C}^\infty([0,T]\times \mathbb{R}^d)$ are defined as in \eqref{eq : seminorms}.
This is done for notational consistency; actually, all estimates take place on the initial compact set $K_0$.
Similarly, we take $B_0$ as starting element of an exhausting sequence $(B_n)_{n\in\N}$ of compact subsets of $\mathbb{R}^d$.
The corresponding seminorms on $\mathcal{C}^\infty(\mathbb{R}^d)$ are denoted by $\mu_n^0$.

We have the following properties of the solution operator $L$ in space dimensions $d = 1,2,3$.
\begin{lemma}\label{lem:waveestimates}
Assume that $u_0, u_1 \in \mathcal{C}^\infty(\mathbb{R}^d)$ and $h \in \mathcal{C}^\infty([0,T]\times\mathbb{R}^d)$.

(a) If the supports of $u_0$ and $u_1$ are contained in $B_0$ and the support of $h$ is contained in $K_0$, then the support of $L(u_0,u_1,h)$ is contained in $K_0$.

(b) There is a sequence of positive constants $C_n$, depending only on $T$, such that
\[
   \mu_n\big(L(u_0,u_1,h)\big) \leq C_n\Big(\mu_{n+1}^0(u_0) + \mu_n^0(u_1) + \mu_n(h)\Big).
\]
\end{lemma}

\begin{proof}
(a) The support property is obvious from the primary versions of d'Alembert's, Poisson's, and Kirchhoff's solution formulas.

(b) The estimate of order $n=0$ is obvious from the secondary versions of the formulas. When differentiating $L(u_0,u_1,h)$ with respect to $x$ or $t$, the alternative versions of the solution formulas are more useful. One easily observes that each derivative of $L(u_0,u_1,h)$ consists of a finite number of terms involving the derivatives of $u_0$, $u_1$ and $h$, the values of
$h$ at $t=0$, and possibly integrals thereof over $B_0$ or $K_0$. The estimate follows immediately.
\end{proof}

\end{document}